\documentclass[12pt]{article}
 \usepackage{a4}                          
\usepackage{amssymb}
\usepackage{amsmath}
\usepackage{amsfonts}
\usepackage{amsthm}

 \newcommand{\eps}{\varepsilon}

\DeclareMathOperator{\sign}{sign}
\newtheorem{thm}{Theorem}[section]
\newtheorem{lem}[thm]{Lemma}

\newtheorem{rem}[thm]{Remark}
\newtheorem{ex}[thm]{Example}

\title{A determinant characterization of moment sequences with finitely many mass-points}

\author{Christian Berg\thanks{The  first author has
    been supported by grant 10-083122 from  {\it The Danish
    Council for Independent Research $|$ Natural Sciences}}
\and Ryszard Szwarc
}
\begin{document}
\maketitle

\begin{abstract} To a sequence $(s_n)_{n\ge 0}$ of real numbers we associate the sequence of Hankel matrices $\mathcal H_n=(s_{i+j}),0\le i,j \le n$. We prove that if the corresponding sequence of Hankel determinants $D_n=\det\mathcal H_n$ satisfy $D_n>0$ for $n<n_0$ while $D_n=0$ for $n\ge n_0$, then all Hankel matrices are positive semi-definite, and in particular $(s_n)$ is the sequence of moments of a discrete measure concentrated in $n_0$ points on the real line. We stress that the conditions $D_n\ge 0$ for all $n$ do not imply the positive semi-definiteness of the Hankel matrices.  
  \end{abstract}


\section{Introduction and results}\label{sec-intro}

Given a sequence of real numbers $(s_n)_{n\ge 0}$, it was proved by Hamburger \cite{Ha} that it can be represented as
\begin{equation}\label{eq:Ha}
s_n=\int_{-\infty}^\infty x^n\,d\mu(x),\quad n\ge 0
\end{equation}
with a positive measure $\mu$ on the real line, if and only if all the Hankel matrices
$\mathcal H_n=(s_{i+j}),0\le i,j\le n$ are positive semi-definite. The sequences \eqref{eq:Ha} are called {\it Hamburger moment sequences} or {\it positive definite sequences} on $\mathbb N_0=\{0,1,\ldots\}$ considered as an additive semigroup under addition, cf. \cite{B:C:R}.

Given a Hamburger moment sequence it is clear that all the Hankel determinants $D_n=|\mathcal H_n|$ are non-negative. It is also easy to see (cf. Lemma~\ref{thm:deg} and its proof)  that only two possibilities can occur:
Either $D_n>0$ for $n=0,1,\ldots$ and in this case any $\mu$ satisfying \eqref{eq:Ha} has infinite support, or there exists $n_0$ such that $D_n>0$ for $n\le n_0-1$ and $D_n=0$ for $n\ge n_0$. In this latter case $\mu$ from \eqref{eq:Ha} is uniquely determined and is a discrete measure concentrated in $n_0$ points on the real axis. (If $n_0=0$ and $D_n=0$ for all $n$, then $\mu=0$ is concentrated in the empty set.)

The purpose of the present paper is to prove the following converse result:

\begin{thm}\label{thm:main} Let $(s_n)$ be a real sequence and assume that the sequence of Hankel determinants $D_n=|\mathcal H_n|$ satisfy $D_n>0,n\le n_0-1$, $D_n=0,n\ge n_0$. Then $(s_n)$ is a Hamburger moment sequence (and then necessarily
 the moments of a uniquely determined measure $\mu$ concentrated in $n_0$ points).
 \end{thm}
  
\begin{rem} {\rm It follows from a general theorem about real symmetric matrices, that if
$D_n>0$ for $n\le n_0$, then the Hankel matrix $\mathcal H_{n_0}$ is positive definite. For a proof see e.g. \cite[p.70]{B:C:R}. On the other hand, one cannot conclude that $\mathcal H_{n_0}$ is positive semi-definite, if it is just known that $D_n\ge 0$ for $n\le n_0$. For the sequence $1,1,1,1,0,0,\ldots$ we have $D_0=D_3=1, D_1=D_2=D_n=0$ for
$n\ge 4$, but the Hankel matrix $\mathcal H_2$ has a negative eigenvalue. It therefore seems to be of interest that Theorem~\ref{thm:main} holds.\footnote{The authors thank Alan Sokal for having mentioned the question.} 
}
\end{rem} 

\begin{rem} {\rm It follows from the proof of Theorem~\ref{thm:main} that the uniquely determined measure $\mu$ is concentrated in the zeros of the polynomial $p_{n_0}$ given by
\eqref{eq:op}.}
\end{rem}

\begin{ex} {\rm Let $a\ge 1$ and define $s_{2n}=s_{2n+1}=a^n,n=0,1,\ldots$. Then the Hankel determinants are $D_0=1,D_1=a-1,D_n=0,\,n\ge 2$. Therefore $(s_n)$ is a moment
sequence of the measure
$$
\mu=\frac{\sqrt{a}-1}{2\sqrt{a}}\delta_{-\sqrt{a}}+ \frac{\sqrt{a}+1}{2\sqrt{a}}\delta_{\sqrt{a}}.
$$
Similarly, for $0\le a\le 1$, $s_0=1, s_{2n-1}=s_{2n}=a^n,n\ge 1$ is a moment sequence of the measure
$$
\mu=\frac{1-\sqrt{a}}2\delta_{-\sqrt{a}}+\frac{1+\sqrt{a}}2\delta_{\sqrt{a}}.
$$
}
\end{ex}

\section{Proofs}\label{sec-pr}

Consider a discrete measure
\begin{equation}\label{eq:disc}
\mu=\sum_{j=1}^n m_j\delta_{x_j},
\end{equation}
where $m_j>0$ and $x_1<x_2<\ldots<x_n$ are $n$ points on the real axis.
Denote the moments 
\begin{equation}\label{eq:simple}
s_k=\int x^k\,d\mu(x)=\sum_{j=1}^n m_jx_j^k, \quad k=0,1,\ldots,
\end{equation}
and let $\mathcal H_k,D_k$ denote the corresponding Hankel matrices and determinants.
The following Lemma is well-known, but for the benefit of the reader we give a short proof.

\begin{lem}\label{thm:deg} The Hankel determinants $D_k$ of the moment sequence \eqref{eq:simple}
satisfy $D_k>0$ for $k<n$ and $D_k=0$ for $k\ge n$.
\end{lem}

\begin{proof} Let 
$$
P(x)=\sum_{j=0}^n a_j x^j
$$
be the monic  polynomial (i.e., $a_n=1$) of degree $n$ with zeros $x_1,\ldots,x_n$. If $\textbf{a}=(a_0,\ldots,a_n)$ then 
$$
\int P^2(x)\,d\mu(x)=\textbf{a}\mathcal H_n \textbf{a}^t=0,
$$
and it follows that $D_n=0$. If $p\ge 1$ and $\textbf{0}_p$ is the zero vector in $\mathbb R^p$, then
also 
$$
(\textbf{a},\textbf{0}_p)\mathcal H_{n+p}(\textbf{a},\textbf{0}_p)^t=0,
$$
and it follows that $D_{n+p}=0$ for all $p\ge 1$.

On the other hand, if a Hamburger moment sequence \eqref{eq:Ha} has $D_k=0$ for some $k$, then there exists $\textbf{b}=(b_0,\ldots,b_k)\in\mathbb R^{k+1}\setminus \{\textbf{0}\}$ such that $\textbf{b}\mathcal H_k=\textbf{0}$. Defining
$$
Q(x)=\sum_{j=0}^k b_jx^j,
$$
we find
$$
0=\textbf{b}\mathcal H_k\textbf{b}^t=\int Q^2(x)\,d\mu(x),
$$
showing that $\mu$ is concentrated in the zeros of $Q$. Therefore $\mu$ is a discrete measure having at most $k$ mass-points. This remark shows that the Hankel determinants of  \eqref{eq:simple} satisfy $D_k>0$ for $k<n$.
\end{proof}

\begin{lem}\label{thm:det1} Consider $n+1$ non-negative integers $0\le c_1<c_2<\ldots<c_{n+1}$, let $p\ge 1$ be an integer and
define the $(n+1)\times(n+p)$-matrix
$$
H_{n+1,n+p}=\left(\begin{matrix} s_{c_1} & s_{c_1+1} &\cdots & s_{c_1+n+p-1}\\
s_{c_2} &  s_{c_2+1} & \cdots & s_{c_2+n+p-1}\\
\vdots &\vdots & \ddots &\vdots\\
s_{c_{n+1}} & s_{c_{n+1}+1} &\cdots & s_{c_{n+1}+n+p-1}
\end{matrix}\right).
$$
For any $(p-1)\times(n+p)$-matrix $A_{p-1,n+p}$  we have 
$$
D=\left|\begin{matrix} H_{n+1,n+p}\\
A_{p-1,n+p}\end{matrix}\right|=0.
$$
\end{lem}

\begin{proof} By multilinearity of a determinant as function of the rows we have
$$
D=\sum_{j_1,\ldots,j_{n+1}=1}^n m_{j_1}\cdots m_{j_{n+1}}x_{j_1}^{c_1}\cdots
x_{j_{n+1}}^{c_{n+1}}\left|\begin{matrix} J\\
A_{p-1,n+p}\end{matrix}\right|,
$$
where $J$ is the $(n+1)\times(n+p)$-matrix with rows
$$
\left(1,x_{j_l},x_{j_l}^2,\ldots, x_{j_l}^{n+p-1}\right),\,l=1,2,\ldots,n+1,
$$
and since there are $n$ points $x_1,\ldots,x_n$, two of these rows will always be equal.
This shows that each determinant in the sum vanishes and therefore $D=0$.
\end{proof}

With $n,p$ as above we now consider a determinant of a matrix
 $(a_{i,j}),0 \le i,j\le n+p$ of size $n+p+1$ of the following special form
$$
M_{n+p}=\left|\begin{matrix}
s_0 & \cdots & s_{n-1} & s_n & \cdots & s_{n+p-1} & s_{n+p}\\
\vdots & \ddots & \vdots & \vdots & \ddots & \vdots & \vdots\\
 s_{n-1} & \cdots &  s_{2n-2} & s_{2n-1} & \cdots & s_{2n+p-2} & s_{2n+p-1}\\
s_n & \cdots & s_{2n-1} & s_{2n} &\cdots & s_{2n+p-1} & x_0\\
s_{n+1} & \cdots & s_{2n} & s_{2n+1} &\cdots & x_1 & a_{n+1,n+p}\\
\vdots & \ddots & \vdots & \vdots & \ddots & \vdots & \vdots\\
s_{n+p} & \cdots & s_{2n+p-1} & x_{p} &\cdots & a_{n+p,n+p-1} & a_{n+p,n+p}
\end{matrix}\right|
$$
which has Hankel structure to begin with, i.e., $a_{i,j}=s_{i+j}$ for $i+j\le 2n+p-1$. For simplicity we have called $a_{n+j,n+p-j}=x_j, j=0,1,\ldots,p$.

\begin{lem}\label{thm:det2}
$$
M_{n+p}=(-1)^{p(p+1)/2}D_{n-1}\prod_{j=0}^p (x_j-s_{2n+p}).
$$
In particular, the determinant is independent of $a_{i,j}$ with $i+j\ge 2n+p+1$.
\end{lem}

\begin{proof} We first observe that the determinant vanishes if we put $x_0=s_{2n+p}$, because then the first $n+1$ rows in $M_{n+p}$ have the structure of the matrix of Lemma~\ref{thm:det1} with $c_j=j-1,j=1,\ldots,n+1$.

Next we develop the determinant after the last column leading to
$$
M_{n+p}= \sum_{l=0}^{n+p} \pm \gamma_l A_l,
$$
where $\gamma_l$ are the elements in the last column and $A_l$ are the corresponding minors, i.e., the determinants obtained by deleting row number $l+1$ and the last column. Notice that
$A_l=0$ for $l=n+1,\ldots,n+p$ because of Lemma~\ref{thm:det1}. Therefore the numbers 
$a_{n+k,n+p}$ with $k=1,\ldots,p$ do not contribute to the determinant.

For $l=0,\ldots,n$ the determinant $A_l$ has the form
$$
\left|\begin{matrix}
s_{c_1} & \cdots & s_{c_1+n} &\cdots & s_{c_1+n+p-1}\\
s_{c_2} & \cdots & s_{c_2+n} &\cdots & s_{c_2+n+p-1}\\
\vdots & \ddots & \vdots & \ddots & \vdots\\
s_{c_n} & \cdots & s_{c_n+n} &\cdots & s_{c_n+n+p-1}\\
s_{n+1} & \cdots & s_{2n+1} & \cdots & x_1\\
\vdots & \ddots & \vdots & \ddots & \vdots\\
s_{n+p} & \cdots & x_p &\cdots & a_{n+p,n+p-1}
\end{matrix}\right|
$$
for integers $c_j$ satisfying  $0\le c_1<\ldots<c_n\le n$.

Each of these determinants vanish for $x_1=s_{2n+p}$ again by Lemma~\ref{thm:det1}, so consequently 
$M_{n+p}$ also vanishes for $x_1=s_{2n+p}$. As above we see that the determinant does not depend on $a_{n+k,n+p-1}$ for $k=2,\ldots,p$.

The argument can now be repeated and we see that $M_{n+p}$ vanishes for $x_k=s_{2n+p}$  when $k=0,\ldots,p$.

This implies that 
$$
M_{n+p}=K \prod_{j=0}^p (x_j-s_{2n+p}),
 $$
where $K$ is the coefficient to $x_0x_1\ldots x_p$, when the determinant is written as
$$
M_{n+p}=\sum_{\sigma} \sign(\sigma) \prod_{j=0}^{n+p}a_{j,\sigma(j)},
$$
and the sum is over all permutations $\sigma$ of $0,1,\ldots,n+p$.

The terms containing the product $x_0x_1\ldots x_p$ requires the permutations $\sigma$ involved to satisfy $\sigma(n+l)=n+p-l,l=0,\ldots,p$. This yields a permutation of $n,n+1,\ldots,n+p$ onto itself reversing the order hence of sign $(-1)^{p(p+1)/2}$, while $\sigma$ yields an arbitrary permutation of $0,1,\ldots,n-1$. This shows that $K=(-1)^{p(p+1)/2}D_{n-1}$.

\end{proof}

\medskip

\noindent{\bf Proof of Theorem~\ref{thm:main}.}

The proof of Theorem~\ref{thm:main} is obvious if $n_0=0$, and if $n_0=1$ the proof is more elementary than in the general case, so we think it is worth  
giving it separately. Without loss of generality we assume $s_0=D_0=1$, and call $s_1=a$. From $D_1=0$ we then get that $s_2=a^2$, and we have to prove that $s_n=a^n$ for $n\ge 3$.

 Suppose now that it has been established that $s_k=a^k$ for $k\le n$, where $n\ge 2$. By assumption we have
\begin{equation}\label{eq:a} 
0=D_n=\left|\begin{matrix} 1 & a & \cdots & a^{n-1} & a^n\\
a & a^2 & \cdots & a^n & s_{n+1}\\
\vdots & \vdots & \ddots & \vdots &\vdots\\
a^{n-1} & a^n & \cdots & s_{2n-2} & s_{2n-1}\\
a^n & s_{n+1} & \cdots & s_{2n-1} & s_{2n}
\end{matrix}\right|.
\end{equation}
Developing the determinant after the last column, we notice that only the first two terms will appear because the minors for the  elements $s_{n+j},j=2,\ldots,n$ have two proportional rows $(1,a,\ldots,a^{n-1})$ and $(a,a^2,\ldots,a^n)$. Therefore
$$
D_n=(-1)^{n+2}a^n\left|\begin{matrix} a & a^2 & \cdots & a^n\\
a^2 & a^3 & \cdots & s_{n+1}\\
\vdots & \vdots & \ddots & \vdots\\
a^{n} & s_{n+1} & \cdots & s_{2n-1}
\end{matrix}\right| + (-1)^{n+3}s_{n+1}
\left|\begin{matrix} 1 & a & \cdots  & a^{n-1}\\
a^2 & a^3 & \cdots & s_{n+1}\\
\vdots & \vdots & \ddots &\vdots\\
a^n & s_{n+1} & \cdots & s_{2n-1}
\end{matrix}\right|,
$$
hence
$$
D_n=(-1)^n(a^{n+1}-s_{n+1})\left|\begin{matrix} 1 & a & \cdots  & a^{n-1}\\
a^2 & a^3 & \cdots & s_{n+1}\\
\vdots & \vdots & \ddots &\vdots\\
a^n & s_{n+1} & \cdots & s_{2n-1}
\end{matrix}\right|.
$$
The last $n\times n$-determinant is developed after the last column and the same procedure as before leads to
$$
D_n=(-1)^{n+(n-1)}\left(a^{n+1}-s_{n+1}\right)^2\left|\begin{matrix} 1 & a & \cdots  & a^{n-2}\\
a^3 & a^4 & \cdots & s_{n+1}\\
\vdots & \vdots & \ddots &\vdots\\
a^n & s_{n+1} & \cdots & s_{2n-2}
\end{matrix}\right|.
$$
Going on like this we finally get
$$
D_n=(-1)^{n+(n-1)+\cdots +2}\left(a^{n+1}-s_{n+1}\right)^{n-1}
\left|\begin{matrix} 1 & a\\ a^n & s_{n+1}\end{matrix}\right|=(-1)^{n(n+1)/2}
\left(a^{n+1}-s_{n+1}\right)^{n},
$$
and since $D_n=0$ we obtain that  $s_{n+1}=a^{n+1}$.

\medskip
We now go to the general case, where $n_0\ge 2$ is arbitrary.

We have already remarked that the Hankel matrix  $\mathcal H_{n_0-1}$ is positive definite, and we claim that $\mathcal H_{n_0}$ is positive semi-definite. In fact, if for $\eps>0$ we define 
\begin{equation}\label{eq:app}
s_k(\eps)=s_k,\;  k\ne 2n_0,\quad s_{2n_0}(\eps)=s_{2n_0}+\eps,
\end{equation}
and denote the corresponding Hankel matrices and determinants $\mathcal H_k(\eps), D_k(\eps)$, then
$$
\mathcal H_k(\eps)=\mathcal H_k,\,0\le k\le n_0-1, D_{n_0}(\eps)=D_{n_0}+\eps D_{n_0-1}=\eps D_{n_0-1}>0.
$$
This shows that $\mathcal H_{n_0}(\eps)$ is positive definite and letting $\eps$ tend to 0 we obtain that $\mathcal H_{n_0}$ is positive semi-definite.

The positive semi-definiteness of the Hankel matrix $\mathcal H_{n_0}$ makes it possible to define a semi-inner product on the vector space $\Pi_{n_0}$ of polynomials of degree
$\le n_0$ by defining $\langle x^j,x^k\rangle=s_{j+k},\,0\le j,k\le n_0$. The restriction of $\langle\cdot,\cdot\rangle$ to  $\Pi_{n_0-1}$ is an ordinary inner product and the formulas 
\begin{equation}\label{eq:op}p_0(x)=1,\;
p_n(x)=\left|\begin{matrix}
s_0 & s_1 & \cdots & s_n\\
\vdots & \vdots &\ddots & \vdots\\
s_{n-1} & s_n & \cdots & s_{2n-1}\\
1 & x & \cdots & x^n
\end{matrix}\right|,\quad 1\le n\le n_0
\end{equation}
define orthogonal polynomials, cf. \cite[Ch. 1]{Ak}.
 While $p_n(x)/\sqrt{D_{n-1}D_n}$ are orthonormal polynomials for $n<n_0$, it is not possible to normalize $p_{n_0}$ since $D_{n_0}=0$.
The theory of Gaussian quadratures remain valid for the polynomials $p_n,n\le n_0$, cf. \cite[Ch.1]{Ak}, so $p_{n_0}$ has $n_0$ simple real zeros and there is a discrete measure
$\mu$ concentrated in these zeros such that
\begin{equation}\label{eq:Gauss}
s_k=\int x^k\,d\mu(x),\quad 0\le k\le 2n_0-1.
\end{equation}

To finish the proof of Theorem~\ref{thm:main} we introduce the moments 
\begin{equation}\label{eq:tilde}
\tilde s_k=\int x^k\,d\mu(x),\quad k\ge 0
\end{equation}
of $\mu$ and shall prove that $s_k=\tilde s_k$ for all $k\ge 0$. We already know this for $k<2n_0$, and we shall now prove that $s_{2n_0}=\tilde s_{2n_0}$. Since $\mu$ is concentrated in the zeros of $p_{n_0}$ we get
\begin{equation}\label{eq:deg}
 \int p^2_{n_0}(x)\,d\mu(x)=0.
\end{equation}

If $(\tilde{D}_k)$
 denotes the sequence of Hankel determinants of the moment sequence $(\tilde{s}_k)$, we get from
 Lemma~\ref{thm:deg} that $\tilde{D}_k=0$ for $k\ge n_0$.

 Developing the determinants $D_{n_0}$ and $\tilde{D}_{n_0}$ after the last column and using that they are both equal to 0, we get
$$
s_{2n_0}D_{n_0-1}=\tilde s_{2n_0}D_{n_0-1},
$$
hence  $s_{2n_0}=\tilde s_{2n_0}$.

Assume now that $s_k=\tilde{s}_k$ for $k\le 2n_0+p-1$ for some $p\ge 1$, and let us prove that $s_{2n_0+p}=\tilde{s}_{2n_0+p}$.

The Hankel determinant $D_{n_0+p}$ is then a special case of the determinant $M_{n_0+p}$ of Lemma~\ref{thm:det2}, and it follows that
$$
D_{n_0+p}=(-1)^{p(p+1)/2}D_{n_0-1}\left(s_{2n_0+p} - \tilde{s}_{2n_0+p}\right)^{p+1}.
$$
Since  $D_{n_0+p}=0$ by hypothesis, we conclude that $s_{2n_0+p}=\tilde{s}_{2n_0+p}$.$\quad\square$

\noindent
Christian Berg\\
Department of Mathematical Sciences, University of Copenhagen,
Universitetsparken 5, DK-2100, Denmark\\
e-mail: {\tt{berg@math.ku.dk}}

\vspace{0.4cm}
\noindent
Ryszard Szwarc\\
Institute of Mathematics,
University of Wroc{\l}aw, pl.\ Grunwaldzki 2/4, 50-384 Wroc{\l}aw, Poland\\ 
e-mail: {\tt{szwarc2@gmail.com}}


\begin{thebibliography}{120}


\bibitem{Ak} N.~I.~Akhiezer, {\it The Classical Moment Problem and Some Related
 Questions in Analysis}. English translation, Oliver and Boyd, Edinburgh,
 1965.

\bibitem{B:C:R} C. Berg, J. P. R.  Christensen and P. Ressel, {\it Harmonic
analysis on semigroups. Theory of positive definite and related
functions}. Graduate Texts in Mathematics vol. {\bf 100}.
Springer-Verlag, Berlin-Heidelberg-New York, 1984.

\bibitem{Ha} H. Hamburger, \"Uber eine Erweiterung des Stieltjesschen
Momentenproblems. Math. Ann. {\bf 81} (1920),  235--319.

\end{thebibliography}
\end{document}